\documentclass[11pt]{amsart}
\usepackage{amssymb}
\usepackage{graphicx}
\usepackage{amsmath}
\usepackage[dvipsnames]{xcolor}
\usepackage{comment}
\usepackage{hyperref}
 
\hypersetup{
    colorlinks=true,
    linkcolor=blue,
    filecolor=magenta,      
    urlcolor=cyan,
    citecolor=blue,
    }

\theoremstyle{plain}
\newtheorem{theorem}{Theorem}[section]
\newtheorem{question}[theorem]{Question}
\newtheorem{lemma}[theorem]{Lemma}

\newtheorem{cor}[theorem]{Corollary}

\newtheorem{claim}[theorem]{Claim}

\theoremstyle{definition}
\newtheorem{definition}[theorem]{Definition}

\newtheorem{example}[theorem]{Example}
\numberwithin{equation}{section}
\newcommand{\field}{\mathbb{Q}}

\title{Monochromatic products and sums in the rationals}
\author{Matt Bowen}

\address{Department of Mathematics and Statistics, McGill University, 805 Sherbrooke St W., H3A 0B9  Montreal, Canada}

\email{matthew.bowen2@mail.mcgill.ca}

\author{Marcin Sabok}

\address{Department of Mathematics and Statistics, McGill University, 805 Sherbrooke St W., H3A 0B9  Montreal, Canada}
\email{marcin.sabok@mcgill.ca}

\thanks{Both authors are partly funded by the NSERC Discovery Grant  RGPIN-2020-05445, NSERC Discovery Accelerator
Supplement  RGPAS-2020-00097, the NCN Grant Harmonia  2018/30/M/ST1/00668 and thank the CRM (Centre de Recherches Math\'emathiques) for its support. The second author thanks also the FSMP (Fondation Sciences Mathématiques de Paris) for its support.}

\begin{document}

\maketitle

\begin{abstract}
    We show that for every coloring of the rationals into finitely many colors, one of the colors contains a set of the form $\{x,y,xy,x+y\}$ for some nonzero $x$ and $y$. 
    

\end{abstract}

\section{Introduction}


The classical theorem of Schur \cite{schur} states that for any finite coloring of the naturals, one of the colors contains a subset of the form $\{x,y,x+y\}$ for some nonzero $x,y$. This result has had quite a few generalizations. For instance, Rado \cite{rado} generalized Schur's theorem to more general families of linear equations and the Folkman theorem \cite[Theorem 11]{grs2} says that for every $r$ we can find a monochromatic set consisting of all sums of subsets of an $r$-element set. 

In \cite{hindman.conjecture} Hindman famously asked (see also the textbook of Graham, Rothschild and Spencer \cite{grs2}) whether any finite coloring of $\mathbb{N}$ contains a monochromatic set of the form $\{x,y,xy,x+y\}$ for some nonzero $x,y$.
Our main result is as follows. 

\begin{theorem}\label{main.simple}
For any coloring of the rationals into finitely many colors there exists a monochromatic set of the form $\{x,y,xy,x+y\}$ for some nonzero $x,y$.
\end{theorem}

In fact, in Theorem \ref{m2}, we prove an extension of the above result involving arithmetic progressions and several variables. In particular, we get
monochromatic sets of the form $\{x,y,xy,x+iy:i\leq k\}$ for any $k\in\mathbb{N}$.


Recently, there has been quite a bit of progress on Hindman's conjecture. 
Green and Sanders \cite{green2016monochromatic}, building on earlier work of Shkredov \cite{shkredov} and Cilleruelo \cite{cilleruelo2012combinatorial}, showed that for any number $n$ of colors there exists a constant $c_n>0$ such that in any $n$-coloring of $\mathbb{F}_p$ there are at least $c_n p^2$ monochromatic tuples of the form $\{x,y,xy,x+y\}$.  In particular, the result of Green and Sanders implies that for any number $n$ there exists a prime $p_n$ such that if $p>p_n$, then in any coloring of $\mathbb{F}_p$ into $n$ colors there exists at least one monochromatic quadruple of the form $\{x,y,xy,x+y\}$. For infinite fields somewhat less was known. Bergelson and Moreira \cite{bergelson2017ergodic} proved that any finite coloring of an infinite field contains a monochromatic set of the form $\{x,xy,x+y\}$ for some nonzero $x,y$.  Later \cite{bergelson2018measure}, they generalized this result to a wider class of rings, and Moreira \cite{moreira2017monochromatic} gave a beautiful proof that any finite coloring of $\mathbb{N}$ contains a monochromatic set of the form $\{x,xy,x+y\}$ for some nonzero $x,y$. 




By a standard compactness argument Theorem \ref{main.simple} extends to other fields, giving the following. 

\begin{cor}\label{finite.fields}
   For every $n$ there exists a prime $p$ such that whenever a field of characteristic at least $p$ is colored with $n$ colours, there exists a monochromatic set of the form $\{x,y,xy,x+y\}$ for some nonzero $x$ and $y$.
\end{cor}

In the case of two colors much more was known. Graham showed that any $2$-coloring of $\{1,\ldots,252\}$ contains a monochromatic configuration the form $\{x,y,xy,x+y\}$, and Hindman showed the same for any $2$-coloring of $\{2,\ldots,990\}$ \cite{hindman.conjecture}. It is worth noting that these proofs were based on a computer search and only recently the first author \cite{bowen2022monochromatic} gave a mathematical proof of the fact that any $2$-coloring of the naturals contains a monochromatic configuration the form $\{x,y,xy,x+y\}$.

In \cite{hindman}  Hindman proved  that  for any coloring of the naturals into finitely many colors, one of the colors contains all sums of finite subsets of an infinite set. A very elegant proof of the latter theorem was given by Galvin and Glazer and their method has since been used to prove a number of other strong combinatorial results (see, e.g., the textbook of Hindman--Strauss \cite{hindman.strauss} or of Todor\v{c}evi\'c \cite{todorcevic}). For instance, Bergelson, Hindman and Leader \cite{bhl} showed how such methods can be used to show that in any measurable coloring of the reals there exists an infinite set all of whose finite products and sums are monochromatic.

The classical van der Warden theorem \cite{vdw} says that for any coloring of the naturals into a finite number of colors, one of those colors contains arbitrarily long arithmetic progressions. This has been famously generalized by the Szem\'eredi theorem \cite{szemeredi} which ensures that arbitrarily long arithmetic progressions can be found in any subset of the naturals of positive density. Since the ergodic-theoretic proof of the Szemer\'edi theorem by {Furstenberg \cite{furstenberg1977ergodic}}, there has been quite a few advances in the field, and many generalizations have been proved using methods coming from ergodic theory (see, e.g., the textbook of McCutcheon \cite{mccutcheon}).

Our proof of Theorem \ref{main.simple} uses two main new ingredients. The first ingredient is a result due to Bergelson and Glasscock \cite{bergelson2016interplay} that is a quantitative version of the Szemer\'edi theorem and is based on the density version of the Hales--Jewett theorem \cite{furstenberg1991density}. The second ingredient seeks to localize certain thick sets with respect to a given finite coloring of the rationals. This ingredient is stated in purely combinatorial terms, in order to make the paper accessible to a wider audience. However, our original motivation was inspired by the methods of Galvin--Glazer and was focused on localizing minimal ideals with respect to the given finite coloring of the rationals.




\section{IP sets and a quantitative version of the Szemer\'edi theorem}


In this section we recall several notions of size that will be useful in the structure of the additive group $(\field,+)$ throughout the remainder of this paper. 

For a finite sequence $(a_1,\ldots,a_n)$ of elements of $\field$ we  use the notation $\sum(a_1,\ldots,a_n)=a_1+\ldots+a_n$. Given a sequence $A$ of elements of $\field$ we write $$\mathrm{FS}(A)=\{ \sum A_0: A_0\mbox{ is finite subsequence of } A\}.$$ 

\begin{definition}
\mbox{}
\begin{itemize}
    \item A subset of $\field$ is \textbf{IP} if it contains a set of the form FS($A$) for some infinite sequence $A$ of elements of $\field$.
    \item Given $r\in\mathbb{N}$, a subset of $\field$ is \textbf{IP}$_r$ if it contains a set of the form FS($A$) for some finite sequence $A$  of length $r$ consisting of elements of $\field$.
    \item A subset of $\field$ is \textbf{IP}$^*_r$ if it has non-empty intersection with every IP$_r$ subset of $\field$.  
\end{itemize}
\end{definition}

Recall that an \textbf{invariant mean} on a commutative semigroup $S$ 
is a positive linear functional $d$ of norm $1$ which is translation invariant on the space of all bounded, real-valued functions on $S$ with the $\|\cdot\|_\infty$ norm. Any commutative semigroup admits an invariant mean.

The Szemer\'edi theorem \cite{szemeredi} states that a subset of $(\mathbb{N},+)$ of positive density contains artitrarily long arithmetic progressions. The set of possible differences of such arithmetic progressions is an IP$^*$ set, as proved by Furstenberg and Katznelson \cite{fk}

The following result of Bergelson and Glasscock \cite[Theorem 7.5]{bergelson2016interplay} is a quantitative strengthening of Szemer\'edi's theorem that we will use in the inductive steps of our construction. 

\begin{theorem}[Bergelson, Glasscock,  \cite{bergelson2016interplay} Theorem 7.5]\label{bergelson}
Let $s\in\mathbb{N}$ and $\alpha>0$ be given. There exists $r\in\mathbb{N}$ and $\alpha'>0$ for which the following holds. For any commutative semigroup $C$ and homomorphisms $\varphi_1,\ldots,\varphi_s:C\to C$, for any invariant mean $d$ on $C$ and $A\subseteq  C$ with $d(A)>\alpha$ the set $$\{c\in C: d(A-\varphi_1(c) \cap \ldots\cap A-\varphi_s(c))>\alpha'\}\quad\mbox{ is IP}_r^*.$$ 
\end{theorem}

We will apply the above result in case $C$ is the additive group $(\field,+)$ and the homomorphisms $\varphi_1,\ldots,\varphi_s$ are of the form $\varphi_i(c)=q_i\cdot c$ for some $q_i\in\field$. The key point in our application is that the numbers $r$ and $\alpha'$ depend only on $s$ and $\alpha$, and not on the particular choice of the homomorphisms $\varphi_1,\ldots,\varphi_s$. 


\section{Localizing thick and syndetic set in finite colorings}\label{sec:thick}


In this section we recall some notions of size in semigroups, which will be stated and used only in terms of the multiplicative group $(\field\setminus\{0\},\cdot)$ throughout the paper. 

Even though this section is concerned with the structure of thick and syndetic sets in $\field\setminus\{0\}$, an equivalent reformulation of the statements below is in terms of localizing  minimal  left ideals in the semigroup $\beta(\field,\cdot)$ and was our original motivation for the approach below. 

\begin{definition}
\mbox{}
    \begin{itemize}
        \item A set $T\subseteq  \field\setminus\{0\}$ is (multiplicatively) \textbf{thick} if for any finite $F\subseteq  \field\setminus\{0\}$ there is an $a\in \field$ with $a\cdot F\subseteq  T$, 
        \item A set $S\subseteq  \field\setminus\{0\}$ is (multiplicatively) \textbf{syndetic}  if there is a finite $F\subseteq  \field\setminus\{0\}$ so that $\field\setminus\{0\}=F\cdot S$.
    \end{itemize}
\end{definition}
Informally, a set is thick if it contains arbitrarily long (multiplicative) intervals, and it is syndetic if it has bounded (multiplicative) gaps.  Note that a set is thick if and only if its complement in $\field\setminus\{0\}$ is not syndetic. 

The only reason why we choose to work with the semigroup $(\field\setminus\{0\},\cdot)$ rather than $(\field,\cdot)$ is because we want to obtain non-zero elements in our main result. We will sometimes abuse the notation slightly and say that a subset of $\field$ is thick or syndetic if its intersection with $\field\setminus\{0\}$ has this property.


\begin{lemma}\label{prod}
Let $k, r, N\in\mathbb{N}$ and $T_1,\ldots,T_k\subseteq  \field\setminus\{0\}$ be thick sets. There are IP$_r$ sets $S_{1,j}\subseteq T_1,\ldots,S_{k,j}\subseteq T_k$ for every $j<N$ such that for any $i\leq j< N$ and for any $l_i,\ldots,l_j\leq k$ we have $$S_{l_i,i}\cdot S_{l_{i+1},i+1}\cdot\ldots\cdot S_{l_j,j}\subseteq  T_{l_i}.$$
\end{lemma}

\begin{proof}
First, observe that 
if $S$ is IP$_r$ and $t\in \field\setminus\{0\}$ then $St$ is IP$_r$.  This implies that any thick set is IP$_r$. Finally, note that if $T$ is thick and $F\subseteq \field\setminus\{0\}$ is finite, then $\{t\in\field\setminus\{0\}: Ft\subseteq  T\}$ is also thick. 
This allows us to define the sets $S_{1,N-j},\ldots,S_{k,N-j}$ inductively for each $j\in \{0,\ldots,N-1\}$.
\end{proof}



The following lemma is the main technical tool that allows us to localize thick sets within the colors in a finite coloring of $\field$. It will be the key ingedient in dealing with an arbitrary number of colors. 

\begin{lemma}\label{cover}
  Let $\field\setminus\{0\}=\bigcup_{i=1}^n C_i$ be a finite coloring.  There is $k\in\mathbb{N}$, index sets $Y_1,\ldots,Y_k\subseteq  [n],$ and a finite set $F\subseteq  \field\setminus\{0\}$ such that 
  
  \begin{itemize}
  
  \vspace{2mm}
      \item[(i)] for each $l \leq k$ the set $\bigcup_{m\in Y_l} C_m$ is thick,
      
      \vspace{2mm}
      
      \item[(ii)] for each $x\in \field\setminus\{0\}$ there exists $l\leq k$ such that for each $m\in Y_l$ we have $x\in F\cdot C_m$.
  \end{itemize}
\end{lemma}

\begin{proof}

For $Y\subseteq  [n]$ we denote $C_Y=\bigcup_{m\in Y}C_m$. Write $$\mathcal{T}=\{Y\subseteq [n]: C_Y \mbox{ is thick}\}$$  and 

$$\mathcal{S}=\{Y\subseteq [n]: \field\setminus\{0\}=\bigcup_{f\in F_Y} fC_Y\mbox{ for some finite }F_Y\subseteq \field\setminus\{0\}\},$$
i.e., $\mathcal{S}=\{Y\subseteq [n]: C_Y$ is syndentic$\}$. Since $\mathcal{S}$ is a finite collection, we can choose a finite set $F\subseteq \field\setminus\{0\}$ which contains all possible finite sets $F_Y$ in the definition of $\mathcal{S}$ above. 

For $x\in \field\setminus\{0\}$ write
$$A_x=\{m\in[n]: \exists f\in F\   x\in f C_m\}.$$ 
 
 Note that by the definition of $\mathcal{S}$ for every $x\in \field\setminus\{0\}$ we have
 \begin{equation}\label{abs.non}
     A_x\cap Y\not=\emptyset\quad\mbox{for every }Y\in\mathcal{S}.
 \end{equation}

\begin{claim}\label{new.abstract.nonsense}
   For every $x\in \field\setminus\{0\}$ there exists $Y_x\in\mathcal{T}$ such that $Y_x\subseteq A_x$.
\end{claim}
\begin{proof}
    For the sake of contradiction suppose that no such set $Y_x$ exists. Then $[n]\setminus A_x$ belongs to $\mathcal{S}$, which contradicts (\ref{abs.non}).
\end{proof}
 
 
 Finally, there are finitely many choices for $Y_x$ as in Claim \ref{new.abstract.nonsense}, and so these choices correspond to the desired $Y_1,\ldots,Y_k$.

\end{proof}

\section{Finding the patterns $\{x,y,xy,x+y\}$}

Before giving the proof of Theorem \ref{main.simple}, we discuss two special cases as a warm-up.
Namely, we first present the proof in the case when each color class is syndetic and next we present the proof in the case when the color classes are all thick.  In the proof of Theorem \ref{main.simple} we will carry out the arguments of these special cases simultaneously, using the ideas of Section \ref{sec:thick}.

\subsection{Two special cases}
While the proof of Theorem \ref{main.simple} will not depend on the next two claims, they might be useful to the reader before moving on to the proof of Theorem \ref{main.simple}.

We begin with the proof in the case that each color class is syndetic. Below, we use Moreira's theorem \cite[Theorem 7.2]{moreira2017monochromatic} in the general form that holds for any field (cf. the remarks after \cite[Definition 7.1]{moreira2017monochromatic}). 

\begin{claim}\label{synd}
Suppose that $\field=\bigcup_{i=1}^n C_i$ and each $C_i$ is syndetic.  There is a monochromatic set of the form $\{x,y,xy,x+y\}$.
\end{claim}

\begin{proof}
By definition, for each $i\in [n]$ there is a finite set $F_i\subseteq\field$ so that $\field\setminus\{0\}=F_i(C_i\setminus\{0\})$.  Let $F=\bigcup_{i=1}^n F_i.$  Then for each $x\in \field\setminus\{0\}$ there is a tuple $(f_1,\ldots,f_n)\in F^n$ so that $x\in f_iC_i$ for each $i\in [n].$  Considering a new coloring of elements of $\field\setminus\{0\}$ based on this tuple, by theorem \cite[Theorem 7.2]{moreira2017monochromatic} we get nonzero $x',y$ for which $\{x',x'y,x'+\frac{y}{f}:f\in F\}$ is monochromatic with color $(f_1,\ldots,f_n)$.  Suppose that $y\in C_j$.  Then letting $x=x'f_j,$ we get $\{x,y,xy,x+y\}\subseteq C_j$ is as desired.
\end{proof}

The other extreme case is when the colors are all thick. For the sake of simplicity we only present the two color case here; the proof extends naturally to more colors, albeit with more complicated notation.

\begin{claim}\label{thick}
Suppose that $\field=C_1\cup C_2$, where both $C_i$ are thick.  There is a monochromatic set of the form $\{x,y,xy,x+y\}.$
\end{claim}

\begin{proof}
Let $d$ be an additively left-invariant mean on $\field$. 
Without loss of generality, we may assume that $d(C_1)>0$. Let $\alpha=d(C_1)$.  
Apply Theorem \ref{bergelson} with $\alpha=d(C_1)$ and $s=1$, to get $r_1$ and $\alpha'$ and apply it again with $\alpha'$ and $s=2$ to get $r_2$.

Since $C_2$ is thick, we can find an IP$_{r_2}$ set $S_2$ contained in $C_2$. Using the fact that $C_1$ is thick, we can find an IP$_{r_1}$ set $S_1$ contained in $C_1$ such that 
\begin{equation}\label{zero}
S_1 S_2\subseteq  C_1
\end{equation}
Now, since $S_1$ is IP$_{r_1}$, applying Theorem \ref{bergelson} to the homomorphism  $x\mapsto x$ we find $C_1'\subseteq  C_1$ with $d(C_1')=\alpha'$ and $y_1\in S_1$ such that 
\begin{equation}\label{first}
C_1'+y_1\subseteq  C_1.    
\end{equation}

Write $D_1=y_1C_1'$. If $D_1\cap C_1\not=\emptyset$, then pick any $x\in C_1'$ such that $xy_1\in C_1$ and put $y=y_1$. Note that (\ref{first}) implies $x+y\in C_1$, so $\{x,y,xy,x+y\}\subseteq C_1$.

Thus, we can assume that $D_1\subseteq  C_2$. Since $S_2$ is IP$_{r_2}$, applying Theorem \ref{bergelson} to the two homomorphisms  $x\mapsto y_1 x$ and $x\mapsto \frac{1}{y_1}x$ we find  $C_1''\subseteq  C_1'$ and $y_2\in S_2$ such that
\begin{equation}\label{third}
    C_1''+y_1 y_2\subseteq  C_1'
\end{equation}
\begin{equation}\label{second}
    C_1''+\frac{y_2}{y_1}\subseteq  C_1'
\end{equation}


Note that $y_1 y_2\in C_1$ by (\ref{zero}). 

If $y_1 y_2 C_1''\cap C_1\not=\emptyset$, then put $y=y_1 y_2$ and choose $x\in C_1''$ such that $xy\in C_1$. Note that (\ref{third}) implies that $x+y\in C_1$. Thus, $\{x,y,xy,x+y\}\subseteq C_1$.

Otherwise, $y_1y_2 C_1''\cap C_2\not=\emptyset$. Put $y=y_2$ and let $x\in y_1 C_1''$ be such that $xy\in C_2$.  Since $y_1 C_1''\subseteq D_1\subseteq C_2$ we have $x\in C_2$. Note that (\ref{second}) implies that $x+y\in C_2$. Thus, in this case $\{x,y,xy,x+y\}\subseteq C_2$.





\end{proof}

In fact, Claims \ref{synd} and \ref{thick} alone can be used to show that any $2$-coloring of $\field$ contains a monochromatic set $\{x,y,xy,x+y\}$ and this was the original motivation of the proof below.

\subsection{The general case}
In this subsection we prove our main result.

\begin{theorem}\label{m2}
For any finite coloring of $\field$ 
there exist  nonzero $y$ and infinitely many $x\in\field$ such that the tuples $\{x,y,xy,x+y\}$ are monochromatic.
\end{theorem}





Note that the above statement, in particular, implies that all elements in the quadruple can be chosen to be distinct. First, we can assure that $y\not=1$ by modifying the coloring by giving $1$ a separate color, and then, using the fact that for a fixed $y\not=1$ the equation $xy=x+y$ has one solution, we can choose $x$ so that all numbers $x,y,xy,x+y$ are distinct.

In the proof we will use Lemma \ref{cover} to carry out the proofs of Claims \ref{synd} and \ref{thick} simultaneously.

\begin{proof}[Proof of Theorem \ref{m2}]
Let 
$d$ be an additive invariant mean on $\field$. Suppose $\field$ is colored into $n$ colors and write $\field\setminus\{0\}= C_1\cup\ldots\cup C_n$. Using Lemma \ref{cover} for the colors $C_i$ find $k\in\mathbb{N}$ and finite sets $F\subseteq \field\setminus\{0\}$ and $Y_1,\ldots,Y_k\subseteq [n]$ 
such that
  \begin{itemize}
  
  \vspace{2mm}
      \item[(i)] for each $l \leq k$ the set $\bigcup_{m\in Y_l} C_m$ is thick,
      
      \vspace{2mm}
      
      \item[(ii)] for each $x\in \field\setminus\{0\}$ there exists $l\leq k$ such that for each $m\in Y_l$ we have $x\in F\cdot C_m$.
  \end{itemize}
This means that for each $x\in \field\setminus\{0\}$ there is $l\leq k$ and a tuple $f_1,\ldots,f_n\in F$ so that  for each $m\in Y_l$ we have $x\in f_m C_m$. 

Note that there are finitely many tuples $(l,f_1,\ldots,f_n)$ with $l\leq k$ and $f_1,\ldots,f_n\in F$ and this gives a new coloring of $\field\setminus\{0\}$ where for each tuple $(l,f_1,\ldots,f_n)$ as above we have a color $D_{(l,f_1, \ldots ,f_n)}$ and 
we put $$x\in D_{(l,f_1,\ldots,f_n)}\quad\mbox{ if }\quad x\in f_m C_m\mbox{ for each } m\in Y_l.$$ Write $K$ for the number of colors in the above coloring. These colors do not need to be disjoint but we can always disjointify them.

Let $N\in\mathbb{N}$ be a large enough. Let $s\in\mathbb{N}$ be a large number, depending on $N$. There exists $r\in\mathbb{N}$ such that we can iteratively apply Theorem \ref{bergelson} $N$ many times as follows. We define a sequence of positive numbers $\alpha_1,\alpha_1'\ldots,\alpha_N,\alpha_N'$ with $\alpha_1=\frac{1}{K}$. At the $j$-th step, given $\alpha_j$ and $s$ we apply Theorem \ref{bergelson} and obtain $\alpha_j'$ such that for any set $A\subseteq \field$ with $d(A)\geq\alpha_j$ and for any $q_1,\ldots,q_s\in\field$ the set
\begin{equation}\label{berg.glas}
    \{y\in \field: d(\{x\in A: x+q_1 y\in A,\ldots, x+q_s y\in A \})>\alpha'\}\mbox{ is IP}_r^*.
\end{equation}
We put $\alpha_{j+1}=\frac{\alpha_j'}{K}$.

Since the sets $\bigcup_{m\in Y_1} C_m,\ldots, \bigcup_{m\in Y_k} C_m$ are thick, by Lemma \ref{prod} there are IP$_r$ sets $S_{1,j}\subseteq \bigcup_{m\in Y_1} C_m,\ldots,S_{k,j}\subseteq \bigcup_{m\in Y_k} C_m$ for each $j<N$ such that 
\begin{equation}\label{products.applied}
S_{l_i,i}\cdot S_{l_{i+1},i+1}\cdot\ldots\cdot S_{l_j,j}\subseteq  \bigcup_{m\in Y_{l_i}}C_m    
\end{equation}
 for any $i\leq j\leq N$ and any choice $l_i,l_{i+1},\ldots,l_j\leq k$

We inductively define
\begin{itemize}
    \item a sequence of subsets $A_1\supseteq\ldots\supseteq A_N$ of $\field$,
    \item finite sets $Q_1,\ldots,Q_N\subseteq  \field$, 
    \item tuples $(l_1,f_{1,1},\ldots,f_{n,1}),\ldots, (l_N,f_{1,N},\ldots,f_{n,N})$ such that  $l_j\leq k$ and $f_{1,j},\ldots,f_{n,j}\in F$ for every $j\leq N$, 
    \item and elements $y_1,\ldots,y_{N-1}\in \field\setminus\{0\}$,
\end{itemize}
such that for every $j< N$ we have 

\begin{enumerate}
    
    \item\label{szemeredi} $A_{j+1}\subseteq \{x\in A_j\cap\bigcap_{q\in Q_j}(A_j-qy_j)\}$,
\item\label{pigeonhole} $A_{j+1}\subseteq \{x\in A_j: xy_1\cdot\ldots\cdot y_j\in D_{(l_{j+1},f_{1,j+1},\ldots,f_{n,j+1})}\}$,
    \item\label{y} $y_j\in S_{l_{j},j}$. 
\end{enumerate}
and $d(A_j)\geq\alpha_j$ for every $j<N$ with  $$Q_j=\{\frac{y_i\cdot\ldots\cdot y_{j-1}}{fy_1\cdot\ldots\cdot y_{i-1}}:  f\in F, 1\leq i<j\}$$ for every $j< N$, where by convention we write $y_1\cdot\ldots\cdot y_{0}=1$ and $Q_1=\{\frac{1}{f}:f\in F\}$.
Note that $|Q_j|$ only depends on $N$ and $F$ and so we can assume that $s\in\mathbb{N}$ is chosen large enough in advance so that $|Q_j|\leq s$ for each $j<N$.

To start off, we choose $(l_1,f_{1,1},\ldots,f_{n,1})$ such that $d(D_{(l_1,f_{1,1},\ldots,f_{n,1})})>\alpha_1$ and put $A_1= D_{(l_1,f_{1,1},\ldots,f_{n,1})}$. 
Such $(l_1,f_{1,1},\ldots,f_{n,1})$ can be found by our choice of $\alpha_1$. 




Now we describe the induction.
Suppose that $A_j$, $Q_j$, $(l_j,f_{1,j}\ldots,f_{n,j})$ as well as $y_{j-1}$ have been constructed. We proceed to find $y_j$ as well as $A_{j+1}$, $Q_{j+1}$ and $(l_{j+1},f_{1,j+1}\ldots,f_{n,j+1})$.

First, since we can assume that $s$ is chosen big enough so that $|Q_j|\leq s$, and since the set $S_{l_j,j}$ is IP$_r$, by (\ref{berg.glas}) we can find $y_j\in S_{l_{j},j}$ and $A_j'\subseteq  A_j$ such that $d(A_j')=\alpha_j'$ and for every $x\in A_j'$ we have $$x+qy_{j}\in A_j \quad\mbox{for all }q\in Q_j.$$ 

Second, look at $y_1\cdot\ldots\cdot y_j A_j'$ and note that for some $(l_{j+1},f_{1,j+1},\ldots,f_{n,j+1})$ the set $\{x\in A_j': xy_1\cdot\ldots\cdot y_j\in D_{(l_{j+1},f_{1,j+1},\ldots,f_{n,j+1})}\}$ has density at least $\alpha_{j+1}$, by our choice of $\alpha_{j+1}$. Put $$A_{j+1}=\{x\in A_j': xy_1\cdot\ldots\cdot y_j\in D_{(l_{j+1},f_{1,j+1},\ldots,f_{n,j+1})} \}.$$


So long as $N$ was chosen large enough, by the pigeonhole principle we find $i<j<N$ so that 
for some $(l,f_1,\ldots,f_n)$ we have
\begin{equation}\label{pig}
    (l_i,f_{1,i},\ldots,f_{n,i})=(l_j,f_{1,j},\ldots,f_{n,j})=(l,f_1,\ldots,f_n).
\end{equation}

Let $y=y_i\ldots y_{j-1}$. By the property (\ref{y}) and (\ref{products.applied}), there exists  $m\in Y_{l}$ such that $y\in C_m$. 
Let $x'\in A_{j}$ and put $x=f_m x' y_1\cdot\ldots\cdot y_{i-1}$.  

Since $A_{j}\subseteq  A_{i}$,  by the property (\ref{pigeonhole}) at step $i$ of the construction and we know that 
$$x' y_1\cdot\ldots\cdot y_{i-1}\in D_{(l,f_1,\ldots,f_n)},$$ so $x=f_m x'y_1\cdot\ldots\cdot y_{i-1}\in C_m$ by (\ref{pig}). 

By property (\ref{pigeonhole}) at step $j$ we know that $$x' y_1\cdot\ldots\cdot y_{j-1}\in D_{(l,f_1,\ldots,f_n)},$$ so $xy=f_m x' y_1\cdot\ldots\cdot y_{j-1}\in C_m$. 

Finally, by property (\ref{szemeredi}) we have $A_{j}\subseteq  A_{j-1}\cap \bigcap_{q\in Q_{j-1}}(A_{j-1}-qy_{j-1})\subseteq  A_{i}\cap \bigcap_{q\in Q_{j-1}}(A_i-qy_{j-1})$, so $$x'+q y_{j-1}\in A_i$$ for every $q\in Q_{j-1}$. Put $$q=\frac{y_{i}\cdot\ldots\cdot y_{j-2}}{f_m y_1\cdot\ldots\cdot y_{i-1}}.$$ Then by (\ref{pigeonhole}) again, since $x'+q y_{j-1}\in A_i$, we have $y_1\cdot\ldots\cdot y_{i-1}(x'+q y_{j-1})\in D_{(l,f_1,\ldots,f_n)}$, i.e., $$x'y_1\cdot\ldots\cdot y_{i-1}+\frac{1}{f_m}y_i\ldots y_{j-1}=\frac{1}{f_m}x+\frac{1}{f_m}y\in D_{(l,f_1,\ldots,f_n)},$$ which means that $x+y\in C_m$. 

\end{proof}


\subsection{Other fields}
Now, we give the short proof of Corollary \ref{finite.fields}, using a standard compactness argument.

\begin{proof}[Proof of Corollary \ref{finite.fields}]
    Write $T_F$ for the theory of fields. Fix $n$ and let $L$ be the language of the theory of fields together with $n$ unary predicates $C_1,\ldots, C_n$. By Theorem \ref{m2} the sentence $$\sigma=\big(\forall x\bigvee_{i=1}^n C_i(x)\big)\to\big(\exists x,y\not=0 \bigvee_{i=1}^n C_i(x)\wedge C_i(y)\wedge C_i(xy)\wedge C_i(x+y)\big)$$ is satisfied in $\mathbb{Q}$ and hence in any field of characteristic zero. Thus, $\sigma$ is provable in the theory of fields of characteristic zero $$T_0=T_F\cup\{\underbrace{1+\ldots+1}_{p}\not=0:p\mbox{ prime}\}.$$ By compactness, $\sigma$ is provable from a finite subset of $T_0$, hence there exists $p_0$ such that any field of characteristic greater than $p_0$ satisfies $\sigma$.
\end{proof}

\section{Generalizations}

Here we prove a generalization of Theorem \ref{main.simple} to the setting of more variables and more complicated monochromatic patterns.

\begin{theorem}
Let $H\subseteq\{h:\field^i\to \field: i\in\mathbb{N}\}$ be a finite set of functions and let $t\in\mathbb{N}$. In any finite coloring of $\field$ there are $x_1,\ldots,x_t\in \field$ such that the following numbers have the same color: 
$$x_i\cdot\ldots\cdot x_j,\quad x_0\cdot\ldots\cdot x_i+
h_{i+1}(x_1,\ldots,x_{i})x_{i+1}+ \ldots + h_t(x_1,\ldots,x_{t-1})x_{t}$$ for all $0\leq i\leq j\leq t$ and any $h_{i+1},\ldots,h_t\in H$ (of appropriate arity) .
\end{theorem}

\begin{example} Considering the constant $0$ and constant $1$ functions together with $h_1(y,z)=y$, $h_2(y,z)=z$ and $h_3(y,z)=yz$ we get the following monochromatic pattern:
\begin{eqnarray*}
    \{x,y,z,xy,yz,xyz, \\  x+y,x+z,x+y+z,x + y + yz, x+yz,\\  xy+z,xy+yz\}.
\end{eqnarray*}

\end{example}
\begin{example} For a given $k$, considering the constant functions $1,\ldots, k$ we get the following monochromatic pattern:
    $$\{x,y,xy,x+iy: i\leq k\}.$$ 
\end{example}

The proof is essentially the same as that of Theorem \ref{main.simple}, but now we will need to consider more complicated sets $Q_i$ in the proof and use Ramsey's theorem together with the pigeonhole principle. We repeat the first part of the proof for the sake of completeness.



\begin{proof}
  
      
      

 Suppose $\field$ is colored into $n$ colors and write $\field\setminus\{0\}= C_1\cup\ldots\cup C_n$. Let $d$ be an additive invariant mean on $\field$. By Lemma \ref{cover} we get $k\in\mathbb{N}$ and finite sets $F\subseteq \field$ and $Y_1,\ldots,Y_k\subseteq [n]$ 
such that
  \begin{itemize}
  
  \vspace{2mm}
      \item[(i)] for each $l \leq k$ the set $\bigcup_{m\in Y_l} C_m$ is thick,
      
      \vspace{2mm}
      
      \item[(ii)] for each $x\in \field\setminus\{0\}$ there exists $l\leq k$ such that for each $m\in Y_l$ we have $x\in F\cdot C_m$.
  \end{itemize}
This gives the coloring of $\field\setminus\{0\}$ with colors $D_{(l,f_1,\ldots,f_n)}$  for each tuple $(l,f_1,\ldots,f_n)$ where
$$x\in D_{(l,f_1,\ldots,f_n)}\quad\mbox{ if }\quad x\in f_m C_m\mbox{ for each } m\in Y_l.$$
These colors do not need to be disjoint but we can always disjointify them. Write $K$ for the number of colors in the above coloring. Let $N,s\in\mathbb{N}$ be large numbers,  $s$ depending on $N$.

 There exists $r\in\mathbb{N}$ such that we can iteratively apply Theorem \ref{bergelson} $N$ many times as follows. We define a sequence of positive numbers $\alpha_1,\alpha_1'\ldots,\alpha_N,\alpha_N'$ with $\alpha_1=\frac{1}{K}$. At the $j$-th step, given $\alpha_j$ and $s$ we apply Theorem \ref{bergelson} and obtain $\alpha_j'$ such that for any set $A\subseteq \field$ with $d(A)\geq\alpha_j$ and for any $q_1,\ldots,q_s\in\field$ the set
\begin{equation}\label{berg.glas.gen}
    \{y\in \field: d(\{x\in A: x+q_1 y\in A,\ldots, x+q_s y\in A \})>\alpha'\}\mbox{ is IP}_r^*.
\end{equation}
We put $\alpha_{j+1}=\frac{\alpha_j'}{K}$.



Lemma \ref{prod} gives us IP$_r$ sets $S_{1,j}\subseteq \bigcup_{m\in Y_1} C_m,\ldots,S_{k,j}\subseteq \bigcup_{m\in Y_k} C_m$ for each $j<N$ such that 
\begin{equation}\label{products.applied.gen}
S_{l_i,i}\cdot S_{l_{i+1},i+1}\cdot\ldots\cdot S_{l_j,j}\subseteq  \bigcup_{m\in Y_{l_i}}C_m   
\end{equation}
 for any $i\leq j\leq N$ and any choice of $l_i,l_{i+1},\ldots,l_j\leq k$ 

Inductively define a sequence of subsets $A_1\supseteq\ldots\supseteq A_N$ of $\field$ as well as finite sets $Q_1,\ldots,Q_N\subseteq  \field$, tuples $(l_1,f_{1,1},\ldots,f_{n,1}),\ldots, (l_N,f_{1,N},\ldots,f_{n,N})$ with $l_j\leq k$ and $f_{1,j},\ldots,f_{n,j}\in F$ for every $j\leq N$, as well as elements $y_1,\ldots,y_{N-1}\in \field\setminus\{0\}$
 such that for every $j\leq N$ we have 


\begin{enumerate}
    
    \item\label{szemeredi.gen} $A_{j+1}\subseteq \{x\in A_j\cap\bigcap_{q\in Q_j}(A_j-qy_j)\}$,
\item\label{pigeonhole.gen} $A_{j+1}\subseteq \{x\in A_j: xy_1\cdot\ldots\cdot y_j\in D_{(l_{j+1},f_{1,j+1},\ldots,f_{n,j+1})}\}$,
    \item\label{y.gen} $y_j\in S_{l_{j},j}$. 
\end{enumerate}
and $d(A_j)\geq\alpha_j$ for every $j\leq N$ and 
\begin{eqnarray*}
   Q_j=\{\frac{h(y_1\cdot\ldots\cdot y_{i_1-1},,\ldots,y_{i_{s-1}}\cdot\ldots\cdot y_{i_s-1})y_{i_s}\cdot\ldots\cdot y_{j-1}}{fy_1\cdot\ldots\cdot y_{i_q-1}}: \\ f\in F, h\in H, s\leq j, 0\leq i_1<\ldots<i_s\leq j, q\leq s\}, 
\end{eqnarray*}

where by convention, we write $y_1\cdot\ldots\cdot y_{0}=1$ and $Q_1=\{\frac{1}{f}:f\in F\}$.


Now we describe the induction.
Suppose that $A_j$, $Q_j$, $(l_j,f_{1,j}\ldots,f_{n,j})$ as well as $y_{j-1}$ have been constructed. We proceed to find $y_j$ as well as $A_{j+1}$, $Q_{j+1}$ and $(l_{j+1},f_{1,j+1}\ldots,f_{n,j+1})$.

First, since we can assume that $s$ is chosen big enough so that $|Q_j|\leq s$, and since the set $S_{l_j,j}$ is IP$_r$, by (\ref{berg.glas}) we can find $y_j\in S_{l_{j},j}$ and $A_j'\subseteq  A_j$ such that $d(A_j')=\alpha_j'$ and for every $x\in A_j'$ we have $$x+qy_{j}\in A_j \quad\mbox{for all }q\in Q_j.$$ 

Second, look at $y_1\ldots y_j A_j'$ and note that for some $(l_{j+1},f_{1,j+1},\ldots,f_{n,j+1})$ the set $\{x\in A_j': xy_1\cdot\ldots\cdot y_j\in D_{(l_{j+1},f_{1,j+1},\ldots,f_{n,j+1})}\}$ has density at least $\alpha_{j+1}$, by our choice of $\alpha_{j+1}$. Put $$A_{j+1}=\{x\in A_j': xy_1\cdot\ldots\cdot y_j\in D_{(l_{j+1},f_{1,j+1},\ldots,f_{n,j+1})} \}.$$










So long as $N$ was chosen large enough, 
by the pigeonhole principle we find $j_1,\ldots,j_M$ so that for some $(l,f_1,\ldots,f_n)$ we have
\begin{equation}\label{pig.gen}
    (l_{j_1},f_{1,j_1},\ldots,f_{n,j_1})=(l_{j_i},f_{1,j_i},\ldots,f_{n,j_i})=(l,f_1,\ldots,f_n)
\end{equation}
for each $i\leq M$. For simplicity of notation suppose $j_1=1,\ldots, j_M=M$.

 Note that if $i<j\leq M$, then by (\ref{y.gen}) and (\ref{products.applied.gen}) there exists $m\in Y_l$ such that $y_i\cdot\ldots\cdot y_{j-1}\in C_m$.  Now, consider the coloring of $[M]\times[M]$ where we give $(i,j)$ with $i<j$ the color $m$ if $$y_{i}\cdot\ldots\cdot y_{j-1} \in C_m.$$  So long as $M$ is large enough, by Ramsey's theorem there is a color $m$ and subsequence  $j_1,\ldots,j_{t+1}$ so that all pairs from the subsequence have the same color. 

Put $x_1=y_{j_1}\cdot\ldots\cdot y_{j_2-1}, \ \ldots,\ x_t=y_{j_{t}}\cdot\ldots\cdot y_{j_{t+1}-1}$. 

First note that by the choice of the subsequence $j_1,\ldots,j_t$  we have $$x_i\cdot\ldots\cdot x_j\in C_m$$ for every $1\leq i\leq j$.

Choose $x_0'\in A_{j_t}$ and put $x_0= f_m x_0' y_1\cdot\ldots\cdot y_{j_1-1}$.
Since $A_{j}\subseteq  A_{i}$ for $i<j$,  
by property (\ref{pigeonhole.gen}) we know that by (\ref{pig.gen}) we have $x_0' y_1\cdot\ldots\cdot y_{j_i-1}\in D_{(l,f_1,\ldots,f_n)}$, so $$x_0 x_1\cdot\ldots\cdot x_i= f_m x_0' y_1\cdot\ldots\cdot y_{j_i-1}\in C_m$$ for each $0\leq i\leq n$. 

Now fix $i$ and $h_{i+1},\ldots,h_t\in H$.
By property (\ref{szemeredi.gen}) we know that if $i<j$, then $A_{j}\subseteq  A_{j-1}\cap \bigcap_{q\in Q_{j-1}}(A_{j-1}-qy_{j-1})\subseteq  A_{i}\cap \bigcap_{q\in Q_{j-1}}(A_i-qy_{j-1})$, so 
\begin{equation}\label{***}
x_0'+q y_{j-1}\in A_i     
\end{equation}
 for every $q\in Q_{j-1}$.

Put $$q_{s}=\frac{h_s(x_1,\ldots,x_{s-1})y_{j_{s}}\cdot\ldots\cdot y_{j_{s+1}-2}}{f_m y_1\cdot\ldots\cdot y_{j_i-1}}\in Q_{j_{s+1}-1}$$ for every $s$ such that $i<s\leq t$. Then by iterating (\ref{***}) we have
$$x_0'+
q_{j_{i}-1}y_{j_{i+2}-1}+\ldots+q_{j_{s}}y_{j_{s+1}-1}
\in  A_{j_s}$$ for every $s\leq t$, so $x_0'+
q_{i+1}y_{j_{i+2}-1}+\ldots+q_{t}y_{j_{t+1}-1}
\in  A_{j_t}$. Since $A_{j_t}\subseteq A_{j_i}$ we get
$y_1\cdot\ldots\cdot y_{j_i-1}(x_0'+
q_{i+1}y_{j_{i+2}-1}+\ldots+q_{t}y_{j_{t+1}-1}
)\in D_{(l,,f_1\ldots,f_n)}$, which means that $f_m y_1\cdot\ldots\cdot y_{j_i-1}(x_0'+
q_{i+1}y_{j_{i+2}-1}+\ldots+q_{t}y_{j_{t+1}-1}
)\in C_m$ and so $$x_0\cdot\ldots\cdot x_i+
h_{i+1}(x_1,\ldots,x_{i})x_{i+1}+ \ldots + h_t(x_1,\ldots,x_{t-1})x_{t}\in C_m$$ as needed. 
\end{proof}

\section{Open problems}

Our main result was a common extension of Hindman's conjecture and the van der Waerden theorem, so it is natural to wonder if our approach might extend to generalize the polynomial van der Waerden theorem as well. 

\begin{question}
Let $P$ be a finite set of integral polynomials.  Does any finite coloring of $\mathbb{Q}$ contain a monochromatic set of the form $\{x,y,xy,x+p(y):p\in P\}?$
\end{question}

This is known to hold even for $\mathbb{N}$ if the $y$ term is dropped \cite{moreira2017monochromatic}.  Our approach would immediately extend to prove this if the IP$^*_r$ polynomial Szemer\'edi theorem were known to be true, but perhaps finding the right quantitative version of the polynomial van der Waerden theorem will suffice.

Another potential direction is to extend our results to generalize the geometric van der Waerden theorem as well.  Currently even the following is open.

\begin{question}
Does every finite coloring of $\mathbb{Q}$ contain a monochromatic set of the form $\{xy,xy^2,x+y\}?$
\end{question}

Finally, let us mention that Hindman also conjectured stronger statements, namely that any finite coloring of $\mathbb{N}$ contains monochromatic sets of the form $FS(A)\cup FP(A)$ for $A\subseteq \mathbb{N}$ arbitrarily large.  Our approach can deal with some subsets of these configurations in $\mathbb{Q}$, but the problem still seems difficult even for $|A|=3$ in $\mathbb{Q}.$

\bibliographystyle{amsalpha} 
\bibliography{2}

\end{document}